\documentclass{amsart}
\usepackage{lmodern}
\usepackage[margin=1in]{geometry}
\usepackage{barz,epigraph}
\author{Michael Barz}

\usetikzlibrary{arrows.meta, decorations.markings}
\DeclareMathOperator{\can}{can}

\DeclareMathOperator{\Dol}{Dol}
\DeclareMathOperator{\GM}{GM}
\renewcommand{\hat}{\widehat}
\DeclareMathOperator{\Hodge}{Hodge}

\title{Non-abelian \(p\)-curvature and a non-abelian Katz's formula}

\begin{document}

\maketitle

\begin{abstract}
		Let \(k\) be a field of characteristic \(p,\) and \(f : X \to S\) a smooth proper morphism of smooth \(k\)-schemes. \emph{Katz's formula} \cite{Katz1972} gives a relationship between the Kodaira--Spencer map of \(f\), and an invariant called the \emph{\(p\)-curvature} of the Gauss--Manin connection associated to \(f.\) Recently, Lam--Litt \cite{lamlitt} proved a variant of Katz's formula in non-abelian Hodge theory, and suggested that it should be possible to give a more conceptual proof of their formula using the stacky approach to \(p\)-adic Hodge theory. 

		In this article, we realize their suggestion, explaining how this rather concrete phenomena observed by Katz \cite{Katz1972} and Lam--Litt \cite{lamlitt} can be explained in a conceptual way using \emph{sheared de Rham stacks}, as developed by Bhatt--Kanaev--Vologodsky--Zhang \cite{bkmvz} and Drinfeld \cite{shearedwitt} (though we prove a slightly different statement than Lam--Litt \cite{lamlitt}). We do not assume the reader has any background in the theory of de Rham stacks, and along the way we give some general constructions related to non-abelian connections which may be of use in other contexts. 
\end{abstract}
\setcounter{tocdepth}{1}
\tableofcontents
\epigraph{But no matter how advanced the system, no matter how precise, unless we have the will to communicate, there's no connection.}{\emph{Dance Dance Dance}, Murakami}

\section{Introduction}

Let \(k\) be a field of characteristic \(p,\) and \(f : X \to S\) a smooth proper morphism of smooth \(k\)-schemes. In this situation, one can consider the relative de Rham cohomology \(\mathcal{H}^n_{\dR}(X/S)\) (viewed as a vector bundle on \(S\)).

This vector bundle \(\mathcal{H}^n_{\dR}(X/S)\) admits much additional structure. For example, it famously admits the \emph{Gauss--Manin connection} 
\[\nabla_{\GM} : \mathcal{H}^n_{\dR}(X/S) \to \Omega^1_{S} \otimes \mathcal{H}^n_{\dR}(X/S),\]
a certain flat connection playing an important role in the study of periods.

The relative de Rham cohomology \(\mathcal{H}^n_{\dR}(X/S)\) also carries \emph{two} natural filtrations, the Hodge filtration and the conjugate filtration. This conjugate filtration is specific to characteristic \(p,\) and we will discuss it in \autoref{sec4}. Katz, in his work \cite{Katz1972} on the Grothendieck--Katz \(p\)-curvature conjecture, discovered a remarkable relationship between how \(\nabla_{\GM}\) acts on the associated gradeds of these two filtrations; more precisely, Katz's formula \cite{Katz1972} relates the \emph{Kodaira--Spencer class} of \(f\) with the \emph{\(p\)-curvature} of \(\nabla_{\GM}.\) Katz's formula plays a crucial role in his proof \cite{Katz1972} of the Grothendieck--Katz \(p\)-curvature conjecture for Gauss--Manin connections.

In this note, we will explain a simple proof of a \emph{non-abelian} variant of Katz's formula. This was inspired by recent work of Lam--Litt \cite{lamlitt}, who proved a non-abelian variant of the Grothendieck--Katz \(p\)-curvature conjecture using (amongst other tools) a non-abelian Katz's formula. Lam--Litt \cite{lamlitt}, in their Remark 3.8.2, suggested that their version of Katz's formula should follow from the theory of \emph{syntomification}; this suggestion inspired our article. 
\begin{remark}
		The non-abelian variant of the Grothendieck--Katz \(p\)-curvature conjecture is generally called the Ekedahl--Shepherd-Barron--Taylor conjecture; this conjecture first appeared in print in Bost's article \cite{bost}. 
\end{remark}

The main novelty of this paper is that, by systematically using the sheared de Rham stack of Bhatt--Kanaev--Mathew--Vologodsky--Zhang \cite{bkmvz} and Drinfeld \cite{shearedwitt}, we can give an especially simple proof of the non-abelian Katz's formula, and simplify the construction of the objects appearing in the formula.

\begin{remark}
		In this note, we work entirely over \(k.\) So, to us a \emph{prestack} will mean a functor \(\Alg_k \to \Gpd,\) and a \emph{stack} will be a prestack obeying fpqc descent.
\end{remark}

\subsection{Non-abelian connections}

As before, let \(S\) be a smooth \(k\)-scheme. In this situation, one can define the \emph{sheared de Rham stack} \((S/k)^{\hat{\dR}}.\) This is a stack (by which we mean fpqc sheaf \(\Alg_k\to\Gpd\)) defined by Bhatt--Mathew--Kanaev--Vologodsky--Zhang \cite{bkmvz} and Drinfeld \cite{shearedwitt} (based on earlier work of Bhatt--Lurie \cite{bhattlurie}, Drinfeld \cite{prismatization}, and Simpson \cite{simpson}) having the remarkable property that vector bundles on \((S/k)^{\hat{\dR}}\) are the same as vector bundles \emph{with flat connection} on \(S.\) 

\begin{remark}
		The notation \((S/k)^{\hat{\dR}}\) indicates that we are taking the de Rham stack \emph{relative to \(k\)}; recall that a flat connection on a vector bundle \(\mathcal{E}\) on \(S\) is a map
		\[\nabla : \mathcal{E} \to \Omega^1_{S/k} \otimes \mathcal{E}\]
		obeying certain axioms. The \(k\) in \((S/k)^{\hat{\dR}}\) is the same as the \(k\) in \(\Omega^1_{S/k}.\) 
\end{remark}
\begin{remark}
		There is a `non-sheared' variant of the sheared de Rham stack, discovered first, which only connections with nilpotent \(p\)-curvature. We will not be using this non-sheared variant anywhere in this note.
\end{remark}

One of the basic insights of Simpson is to try and study non-linear differential equations using \emph{stacks over \((S/k)^{\hat{\dR}}\)} instead of studying quasicoherent sheaves on \((S/k)^{\hat{\dR}}.\) 
\begin{defn} \label{defn65}
		Let \(E\) be a prestack with a morphism \(E \to S.\) A \emph{non-abelian connection} on \(E/S\) is a prestack \(E'\) together with a Cartesian diagram
		\begin{center}
				\begin{tikzcd}
						E \ar[r] \ar[d] & E' \ar[d] \\
						S \ar[r] & (S/k)^{\hat{\dR}}.
				\end{tikzcd}
		\end{center}
\end{defn}	

\subsection{The non-abelian Gauss--Manin connection} 

Let \(f : X \to S\) be a smooth proper morphism of smooth \(k\)-schemes, as before. 

We now introduce the non-abelian analogue of the Gauss--Manin connection. To start, we replace the relative de Rham cohomology \(\mathcal{H}^q_{\dR}(X/S)\) (a vector bundle on \(S\)) with the de Rham moduli stack of \(X\) over \(S.\) 
\begin{defn}
		Fix \(n\geq 1.\) We define \(\mathcal{M}_{\dR}(X/S, n)\) to be the mapping stack \(\Map_S((X/S)^{\hat{\dR}}, B\SL_n).\) In other words, \(\mathcal{M}_{\dR}(X/S, n)\) is the \(S\)-stack such that, for \(T \to S\) an \(S\)-scheme, a \(T\)-valued point of \(\mathcal{M}_{\dR}(X/S, n)\) is just a rank \(n\) vector bundle \(\mathcal{E}\) on \(X\times_S T,\) together with a flat connection 
		\[\nabla : \mathcal{E} \to \Omega^1_{(X\times_S T)/T} \otimes \mathcal{E}\]
		and an isomorphism \(\xi : \det(\mathcal{E}) \heq \OO_{X\times_S T}\) \emph{which is horizontal} (that is, it intertwines the connection \(\det(\nabla)\) on \(\det\mathcal{E}\) with the trivial connetion \(d\) on \(\OO\)). 
\end{defn}
\begin{remark}
		We study \(\SL_n\)-bundles rather than \(\GL_n\)-bundles to better match with Lam--Litt \cite{lamlitt}, though for Katz's formula the difference between \(\SL_n\) and \(\GL_n\) does not seem to matter.
\end{remark}

In the body of the paper, we will show how to use de Rham stacks to construct a non-abelian connection on \(\mathcal{M}_{\dR}(X/S, n).\) This connection is furnished by the Cartesian diagram
\begin{center}
		\begin{tikzcd}
				\mathcal{M}_{\dR}(X/S, n) \ar[r] \ar[d] & \Map_{(S/k)^{\hat{\dR}}}((X/k)^{\hat{\dR}}, B\SL_n) \ar[d] \\
				S \ar[r] & (S/k)^{\hat{\dR}}.
		\end{tikzcd}
\end{center}
as we discuss in \autoref{sec3}.

In the body of this paper, we will explain how to construct non-abelian variants of the Hodge and conjugate filtrations on the non-abelian Gauss--Manin connection. 

We can now, somewhat vaguely, state our main theorem; it will be stated more precisely in \autoref{katz}.
\begin{theorem}[Katz's formula, imprecise statement]
		The non-abelian Hodge filtration on \(\mathcal{M}_{\dR}(X/S, n)\) gives rise to a non-abelian Higgs field 
		\[\Theta_{X/S} : \pi^*_{\operatorname{Dol}}T_{S/k}(-1) \to \mathcal{T}_{\mathcal{M}_{\operatorname{Dol}}/(S \times B\G_m)},\] 
		where \(\mathcal{M}_{\operatorname{Dol}} := \Map_{S\times B\G_m}((X/S)^{\operatorname{Hodge}}, B\SL_n),\) with \(\pi_{\operatorname{Dol}} : \mathcal{M}_{\operatorname{Dol}} \to S\times B\G_m\) the structure map  Roughly, this \(\Theta_{X/S}\) is obtained by taking the associated graded of the non-abelian Hodge filtration. 

		Let \(S'\) denote the Frobenius twist of \(S\) relative to \(k.\) Then the non-abelian conjugate filtration on \(\mathcal{M}_{\dR}(X/S, n)\) similarly gives rise to a certain stack \(\pi_{\operatorname{Dol, c}} : \mathcal{M}_{\operatorname{Dol, c}} \to S \times B\G_m,\) analogous to the Dolbeault moduli stack, and a morphism  
		\[\psi_{X/S}|_{\lambda = 0} : \pi_{\operatorname{Dol, c}}^*F_{S/k}^*T_{S'/k}(+1) \to \mathcal{T}_{\mathcal{M}_{\operatorname{Dol, c}}/(S\times B\G_m)}.\] 

		There is a Cartesian diagram
		\begin{center}
				\begin{tikzcd}[column sep=large]
						\mathcal{M}_{\operatorname{Dol, c}} \ar[r, "\phi"] \ar[d] & \mathcal{M}_{\operatorname{Dol}} \ar[d] \\
						S \times B\G_m \ar[r, "{(F_{\abs}, [-1])}"] & S \times B\G_m,
				\end{tikzcd}
		\end{center}
		and for the map \(\phi : \mathcal{M}_{\Dol, c} \to \mathcal{M}_{\Dol}\) in this diagram, we have
		\[\psi_{X/S}|_{\lambda = 0} = \phi^*\Theta_{X/S}.\] 
\end{theorem}
\begin{warn}[Comparison to Lam--Litt \cite{lamlitt}]
		Lam--Litt \cite{lamlitt} prove a related result in their Theorem 3.6.4, though their work uses different constructions of the \(\psi_{X/S}|_{\lambda=0}\) and \(\Theta_{X/S}\) than we do, so this result does not immediately reprove their Theorem 3.6.4 unless one also had a comparison between our definitions. 

		Lam--Litt \cite{lamlitt} use their Theorem 3.6.4 as an intermediate result in establishing their `Theorem A,' which they apply to the non-abelian \(p\)-curvature conjecture. It would be interesting to see how much of Theorem A can be reproven using the theory of de Rham stacks, though Lam--Litt use quite different techniques.
\end{warn}
\begin{remark}
		In Lam--Litt's \cite{lamlitt} approach, they first construct an object \(\Psi_{X/S},\) related to our \(\psi_{X/S}\) by the equality \(\Psi_{X/S} = \lambda\psi_{X/S}.\) One convenience to using de Rham stacks is we construct \(\psi_{X/S}\) right away, whereas Lam--Litt \cite{lamlitt} first need to construct \(\Psi_{X/S},\) and then explain why it is sensible to form \(\frac{1}{\lambda}\Psi_{X/S}\) (which they can only do only a certain locus in \(\mathcal{M}_{\dR}(X/S, n)\) where \(\lambda\) is torsion-free). 
\end{remark}
\begin{remark}[Signs]
		Note that our comparison formula involves a \([-1] : B\G_m \to B\G_m.\) This is related to how, in the abelian setting, Katz's formula invovles a sign of \(-1\) (Katz originally found a different sign due to a small algebra error; see remark 11.4.2 of André's \cite{andre2004}).

		One advantage of the de Rham stack approach is that the sign \(-1\) is forced on us, as without it is impossible to compare the \(+1\) twist in \(\psi_{X/S}|_{\lambda = 0}\) with the \(-1\) twist on in \(\Theta_{X/S}.\) Drinfeld, in formula (4.3) of \cite{drinfeldkatz}, also encounters this sign when comparing \(p\)-curvature and the Kodaira--Spencer class. 
\end{remark}

\subsection{Connection to the theory of synotomification} \label{secsyn}

Recall that vector bundles on \(\A^1/\G_m\) can be identified with filtered vector spaces, via the Rees construction.

The de Rham cohomology of \(S\) carries two natural filtrations in characteristic \(p\): the Hodge filtration and the conjugate filtration. Thus, in the stacky approach to \(p\)-adic Hodge theory, one can view \(H^q_{\dR}(S/k)\) as a vector bundle on \((\A^1/\G_m) \sqcup (\A^1/\G_m).\) Katz's formula ultimately results from the fact that the associated gradeds of these two pieces coincide, up to a Frobenius twist. Ignoring the Frobenius twist for expository purposes, we can then say that \(H^q_{\dR}(S/k)\) actually lives as a vector bundle on 
\[(\Spec k[u, t]/(ut))/\G_m,\]
where the \(\G_m\) acts on \(u\) with degree \(-1\) and on \(t\) with degree \(+1.\) 

\begin{figure}[h]
		\centering
		
		\begin{tikzpicture}[
  axis/.style={thick, -{Stealth[length=6pt]}},
  every node/.style={font=\small}
]
 
 
\draw[axis] (-1.5, 1) -- (1.5, 1) node[right] {$t$};
\filldraw (0,1) circle (2pt) node[below] {$0$};
 
\draw[axis] (-1.5,-1) -- (1.5,-1) node[right] {$u$};
\filldraw (0,-1) circle (2pt) node[below] {$0$};
 
\draw[-{Stealth[length=8pt]}, thick, shorten <=4pt, shorten >=4pt]
  (2.0, 0) -- node[above, align=center, font=\small]
  {glue} (4.0, 0);
 
 
\begin{scope}[xshift=6cm]
  \draw[axis] (-1.5, 0) -- (1.5, 0) node[right] {$t$};
  \draw[axis] (0, -1.5) -- (0, 1.5) node[above] {$u$};
  \filldraw (0,0) circle (2pt) node[below left] {$0$};
\end{scope}
 
\end{tikzpicture}

		\caption{Glueing the associated gradeds (ignoring the Frobenius twist)}
\end{figure}
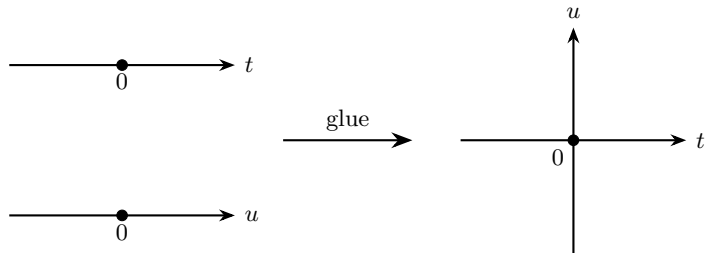

The observation that this glueing can be done is what will lead to our non-abelian Katz's formula. While we do not need it in this paper, we remark that there is one more observation that can be made: not only are the associated gradeds of the Hodge and conjugate filtrations the same, but they are filtrations on the same underlying vector space. Thus, not only can we glue together the closed points of the two copies of \(\A^1/\G_m,\) but we can also glue together their open points. Doing this precisely, and in particular taking into account the Frobenius twist, leads to the theory of \emph{syntomification}.\footnote{To do this, one should also allow \(p\)-adic coefficients and restrict to the case where \(k\) is perfect.} The interested reader can consult chapter 4 of Bhatt's \cite{bhatt} for more information on syntomification; we will not use it in this article. 

\subsection{The structure of this paper}

We hope to make this article readable even to a reader unfamilar with sheared de Rham stacks; thus, we start with the (purely expository!) \autoref{sec2}, which recalls basic facts about sheared de Rham stacks which we will use in the rest of the note.

With this background in place, we get to the main body of the paper. Katz's formula relates two objects: the Hodge filtration and the conjugate filtration. In \autoref{sec3}, we explain the non-abelian Hodge filtration and the construction of the non-abelian Higgs field \(\Theta_{X/S}.\) In \autoref{sec4}, we introduce the conjugate filtration, and construct the non-abelian \(p\)-curvature \(\psi_{X/S}|_{\lambda = 0}\) appearing in Katz's formula. Finally, in \autoref{katz}, with all the ingredients laid out, we prove Katz's formula.

\begin{remark}
		Our approach to non-abelian connections seems related to recent work of Sheng \cite{maosheng} and Fu--Sheng \cite{fusheng}, on what Sheng calls \emph{nonlinear Hodge theory}.
\end{remark}

\subsection{Acknowledgements}

The author thanks Bhargav Bhatt for very helpful conversations on the contents of this article. The author also thanks Zachary Berens and Akshay Venkatesh, who gave comments improving the exposition of this article. 

This paper owes its existence to Joshua Lam and Daniel Litt, for their inspiring article \cite{lamlitt} -- which suggested that de Rham stacks could shed light on non-abelian Katz's formula. The author also had a very helpful correspondence with them on the ideas of this note. We strongly suggest any reader interested in non-abelian \(p\)-curvature explore the fascinating work of Lam--Litt \cite{lamlitt} if they have not already.

\section{Review of sheared de Rham stacks} \label{sec2}

As in the introduction, fix \(k\) a field of positive characteristic, and \(S\) a smooth \(k\)-scheme. In this section, we introduce the \emph{sheared} de Rham stack \((S/k)^{\hat{\dR}}.\) 

\subsection{The characteristic 0 story}

We start by recalling some work of Simpson \cite{simpson} in characteristic 0. For a complex algebraic variety \(Y/\C,\) Simpson \cite{simpson} defined a functor
\[(Y/\C)^{\hat{\dR}} : \Alg_{\C} \to \Set\]
by the formula \((Y/\C)^{\hat{\dR}}(A) := Y(A_{\red}),\) where \(A_{\red} := A/\nil(A)\) for \(\nil(A)\) the nilradical of \(A.\)  

This functor is an fppf sheaf, has the remarkable property that it turns algebraic de Rham cohomology into coherent cohomology: the coherent cohomology of the fppf space \((Y/\C)^{\hat{\dR}}\) is canonically identified with the algebraic de Rham cohomology of \(Y,\) and the category \(\Qcoh((Y/\C)^{\hat{\dR}})\) is canonically identified with the category of quasi-coherent \(\mathcal{D}_Y\)-modules. As linear algebra is easier than differential algebra, in many situations it is convenient to switch from differential algebra on \(Y\) to linear algebra on \((Y/\C)^{\hat{\dR}}\); this is especially the case in non-abelian Hodge theory, as we shall see in this note. 

\begin{remark} \label{rem203}
		Note that \((Y/\C)^{\hat{\dR}}(A)\) is the set of all \(A\)-valued points of \(Y,\) modulo the equivalence relation that two points are equal if they are \emph{infinitesimally close} (meaning that they become the same when restricting to \(A_{\red}\)). 

		Roughly, one can think the importance of nilpotence is because of formal power series: when \(x' - x\in \nil(A)\) is nilpotent, any formal power series of the form \(\sum_{n=0}^{\infty} a_nT^n\) can be evaluated at \(T = x' - x,\) whereas normally infinite sums in a ring do not make sense without first imposing a topology on the ring and discussing convergence. 
\end{remark}
\begin{remark}
		The proof of \autoref{art217} also makes this analogy with power series precise; it is stated in positive characteristic, but the same proof strategy works (with even more ease) in characteristic 0.
\end{remark}

\subsection{The sheared de Rham stack of \(\A^1\)}

We now return to \emph{positive} characteristic.
\begin{warn}[Warning for experts]
		Some readers may be familiar with an object \((S/k)^{\dR},\) called the de Rham stack of \(S\) relative to \(k,\) which geometrizes \(\mathcal{D}_{S/k}\)-modules with \emph{nilpotent \(p\)-curvature}. We will \emph{not} be using this variant of the de Rham stack in this note, and in particular we have no restriction on \(p\)-curvature in our statements. The author regards it as something of a historical accident that this object was discovered before the sheared de Rham stack.
\end{warn}

Inspired on Simpson's work \cite{simpson} in characteristic 0, Drinfeld \cite{shearedwitt} and Bhatt--Kanaev--Mathew--Vologodsky--Zhang \cite{bkmvz} defined a positive characteristic object \((S/k)^{\hat{\dR}},\) called the \emph{sheared de Rham stack}, with the following properties:
\begin{enumerate}
		\item \((S/k)^{\hat{\dR}}\) is an fpqc stack \(\Alg_k \to \Gpd,\) 
		\item quasicoherent sheaves on \((S/k)^{\hat{\dR}}\) are just \(\mathcal{D}_{S/k}\)-modules,
		\item the algebraic de Rham cohomology of \(S/k\) is canonically identified with the coherent cohomology of \((S/k)^{\hat{\dR}}.\) 
\end{enumerate}
\begin{remark}
		We refer a reader to the upcoming work Bhatt--Kanaev--Mathew--Vologodsky--Zhang \cite{bkmvz} for proofs of these claims; we will not use them in our proof of Katz's formula, but we state these claims for motivational purposes. Similarly, we will use that vector bundles on \((S/k)^{\hat{\dR}}\) are simply vector bundles with flat connection on \(S\) when explaining a more concrete interpretation of our version of Katz's formula, but the proof of our main theorem is purely formal and does not rely on understanding this.
\end{remark}

The reader might observe now that Simpson's \((Y/\C)^{\hat{\dR}}\) was a \emph{space} (that is, its functor of points is set valued), whereas the sheared de Rham stack is a \emph{stack} -- its functor of points is groupoid valued. 

We now explain the origin of this discrepancy and give the definition of sheared de Rham stacks. In \autoref{rem203}, we explained that \((Y/\C)^{\hat{\dR}}\) can be obtained by quotienting \(Y\) be the relation of `infinitesimal closeness,' which is useful because one can evaluate arbitrary formal power series on nilpotent elements of a ring. However, the formal power series which arise in calculus often take the form \(\sum_n a_n\frac{(x_2-x_1)^n}{n!}.\) In characteristic 0, these denominators can be ignored, but in characteristic \(p\) it is a big deal. 

Thus, instead of saying that \(x_1, x_2\) are close if \(x_2 - x_1\) is nilpotent, we now say \(x_1\) and \(x_2\) are close if \(x_2 - x_1\) admits \emph{nilpotent divided powers}. 
\begin{defn}
		Let \(A\) be a ring. An element \(a \in A\) is said to admit \emph{nilpotent divided powers} if there is a sequence \((a_0, a_1, a_2, a_3, ...)\) of elements of \(A\) such that 
		\begin{enumerate}
				\item \(a_0 = 1,\)
				\item \(a_1 = a,\)
				\item for all \(n, m,\) we have \(\binom{n+m}{n}a_{n+m} = a_na_m,\)
				\item for all \(n \gg 0,\) we have \(a_n = 0.\)
		\end{enumerate}
\end{defn}
\begin{remark}
		We think of \(a_n\) as being \(a^n/n!\); the relationship \(\binom{n+m}{n}a_{n+m} = a_na_m\) is then because
		\[\frac{a^n}{n!} \cdot \frac{a^m}{m!} = \frac{a^{n+m}}{n!m!} = \frac{a^{n+m}}{(n+m)!} \cdot \binom{n+m}{n}.\] 
\end{remark}

While being nilpotent is a \emph{property} of an element, having nilpotent divided powers is \emph{structure}: there are possibly different nilpotent divided power structures on an element \(a \in A.\) Thus, instead of taking the set of all \(A\)-valued points of \(S\) and quotienting by an equivalence relation, we really should be using a groupoid, to remember that elements can be the same for multiple reasons. 

\begin{exmp}[The affine line] \label{exmp138} 
		Suppose that \(S = \Spec k[x]\) is the affine line. Then \((S/k)^{\hat{\dR}} : \Alg_k \to \Gpd\) is \emph{almost} described by the following formula: send a \(k\)-algebra \(A\) to the groupoid \(\mathcal{G}(A),\) where
		\begin{enumerate}
				\item \emph{objects} of \(\mathcal{G}(A)\) are elements \(x\in A,\)
				\item an \emph{isomorphism} \(x \heq y\) in \(\mathcal{G}(A)\) is the data of a nilpotent divided power structure on the difference \(y-x.\)
		\end{enumerate}
		
		We say \emph{almost} because the functor \(A \mapsto \mathcal{G}(A)\) does not obey fpqc descent, and so \((S/k)^{\hat{\dR}}\) is actually a sheafification of this functor. However, for a \emph{semiperfect} \(k\)-algebra \(A\), this groupoid \(\mathcal{G}(A)\) is the correct groupoid of \(A\)-valued points of \((\A^1/k)^{\hat{\dR}}.\) As semiperfect \(k\)-algebras form a basis in the fpqc topology, in practice one doesn't often need to think about this sheafification.
\end{exmp}

\begin{defn} \label{defn260} 
		We define \(\G_a^{\#, \wedge} : \Alg_k \to \Ab\) the functor
		\[\G_a^{\#, \wedge}(A) = \{(a_0, a_1, ...) \mid a_{\bullet}\text{ is a nilpotent PD sequence}\}.\] 
		The group structure is given by
		\[a_{\bullet} + b_{\bullet} := (1, a_1 + b_1, a_2 + a_1b_1 + b_2, a_3 + a_2b_1 + a_1b_2 + b_3, ...).\] 
		
		Note that there is a natural action of \(\G_a^{\#, \wedge}\) on \(\G_a,\) given on points by
		\[(a_0, a_1, ...) \times b := a_1b.\]
\end{defn}
\begin{theorem} \label{art217}
		Let \(M\) be a \(k[x]\)-module. Giving a \(\G_a^{\#, \wedge}\)-equivariant structure on \(M\) is equivalent to giving a connection \(\nabla : M \to M dx.\) 
\end{theorem}	
\begin{proof}
		Consider the ring \(R' = k[x_1, x_2, (x_2-x_1)^2/2!, ...],\) where we adjoined a divided power sequence for \(x_2 - x_1.\) Set \(R = \inverselim_n R'/((x_2-x_1)^n/n!),\) so that \(R\) is a topological ring whose elements are formal power series of the form
		\[\sum_n p_n(x_1) \cdot \frac{(x_2 - x_1)^n}{n!}.\]
		This topological ring \(R\) is the universal source of nilpotent PD sequences, in the sense that for any \(k\)-algebra \(A,\) giving a pair \(a_1, a_2 \in A\) with a nilpotent PD sequence on \(a_2 - a_1\) is equivalent to giving a continuous ring homomorphism \(R \to A.\) 

		Thus \(\G_a^{\#, \wedge} = \Spf(R),\) so that a \(\G_a^{\#, \wedge}\)-equivariant structure on \(M\) is equivalent to giving an isomorphism
		\[\eta : R \hat{\otimes}_{k[x]}^{x=x_2} M \heq R \hat{\otimes}_{k[x]}^{x=x_1} M.\]
		This \(\eta\) takes the form
		\[\eta(1 \otimes s) = \sum_{n=0}^{\infty} \frac{(x_2-x_1)^n}{n!} \otimes \nabla_n(s),\]
		for certain set-theoretic functions \(\nabla_n : M \to M,\) where the cocycle condition on \(\eta\) forces \(\nabla_n\nabla_m = \nabla_{n+m}, \nabla_0 = \id.\) In particular, \(\eta\) is uniquely determined by \(\nabla_{d/dx} := \nabla_1.\) Observe
		\begin{align*}
				x_2 \otimes s + x_2(x_2 - x_1) \otimes \nabla_{d/dx}(s) + \cdots &= x_2\eta(1 \otimes s) \\
																		  &= \eta(x_2 \otimes s) \\
																		  &= \eta(1 \otimes xs) \\
																		  &= 1 \otimes xs + (x_2 - x_1) \otimes \nabla_{d/dx}(xs) + \cdots \\
																		  &= x_1 \otimes s + (x_2 - x_1) \otimes \nabla_{d/dx}(xs) \\
																		  &= x_2 \otimes s + (x_2 - x_1) \otimes (\nabla_{d/dx}(xs) - s).
		\end{align*}
		Passing to the quotient where \((x_2-x_1)^2/2! = 0,\) we deduce \(\nabla_{d/dx}(xs) = s + \nabla_{d/dx}(s).\) In other words, \(\nabla\) always obeys the Leibniz rule and hence defines a connection; it's easy to see any connection gives rise to a valid \(\eta,\) so we conclude.
\end{proof}

\subsection{Sheared de Rham stacks in general}

There are two approaches to the definition of the sheared de Rham stack of a general smooth \(k\)-scheme \(S.\) Perhaps the most obvious route is to quotient \(S\) by the nilpotent PD-hull of the diagonal of \(S\times_k S.\) However, there is an extremely technically convenient alternative route, via a procedure called `transmutation.' 

Recall Simpson's de Rham stack was defined by \((Y/\C)^{\hat{\dR}}(A) = Y(A/\nil(A)).\) In particular, to define the de Rham stack, one only needs to understand the functor \(A \mapsto A/\nil(A)).\) 

In positive characteristic, as we have already seen, instead of quotienting by the ideal of nilpotent elements, we will need to quotient by some sort of groupoid of nilpotent PD structures. This is mostly simply done using Drinfeld's notion of \emph{quasi-ideals}. 
\begin{remark}[Review of Drinfeld's quasi-ideals]
		We review a few ideas of Drinfeld \cite{drinfeld}. Let \(R\) be a ring. A \emph{quasi-ideal} is an \(R\)-module \(I\) together with a map \(d : I \to R\) such that \(d(x) \cdot y = d(y) \cdot x\) for all \(x, y \in I.\) 

		In this situation, Drinfeld constructs a \emph{ring groupoid} \(\Cone(I \xto{d} R),\) which is a homotopical enhancement of the ring \(R/\im(d).\) 
\end{remark}
\begin{defn}
		Consider the quasi-ideal \(\G_a^{\#, \wedge} \to \G_a\) of \autoref{defn260}. We define a ring stack \(\G_a^{\hat{\dR}}\) as \(\Cone(\G_a^{\#, \wedge} \to \G_a).\) 
\end{defn}
\begin{remark}
		We refer a reader to section 2 of \cite{mine} for discussion of how the Drinfeld cone interacts with sheafification, as \(\G_a\) is not quite a ring but a sheaf of rings. For \(A\) semiperfect, though, it is literally true that
		\[\G_a^{\hat{\dR}}(A) = \Cone(\G_a^{\#, \wedge}(A) \to A).\] 
\end{remark}

We now define the sheared de Rham stack, using essentially the same definition as Simpson but replacing \(A_{\red}\) with \(\G_a^{\hat{\dR}}(A).\) 
\begin{defn}
		The \emph{sheared de Rham stack} of a smooth \(k\)-scheme \(S\) is the fpqc sheaf
		\[(S/k)^{\hat{\dR}} : \Alg_k \to \Gpd\]
		defined by 
		\[(S/k)^{\hat{\dR}}(A) = S(\G_a^{\dR}(A)).\] 
\end{defn}

\begin{remark}[Relative sheared de Rham stacks] \label{rem139}
		For a general smooth morphism \(f : X \to S,\) we define the \emph{relative sheared de Rham stack}
		\begin{equation} \label{eq149} (X/S)^{\hat{\dR}} \heq (X/k)^{\hat{\dR}} \times_{(S/k)^{\hat{\dR}}} S.\end{equation} 
\end{remark}
\begin{remark}
		One can still define \((X/S)^{\hat{\dR}}\) or \((S/k)^{\hat{\dR}}\) without assuming smoothness, but in the general case it seems these objects are not so well-behaved.
\end{remark}

\subsection{On \(p\)-curvature}

When constructing the sheared de Rham stack of \(\A^1,\) we discussed how `admitting divided powers' is a structure instead of a property, and so it is more natural to quotient by a groupoid instead of an equivalence relation. But what if we had just quotiented by the equivalence relation of `admitting nilpotent divided powers'? This ends up resulting in an invariant called \emph{\(p\)-curvature}. A slightly different approach to \(p\)-curvature via de Rham stacks can be found in section 3 of \cite{mine}. 

To begin, we remark that the equivalence relation `admitting nilpotent divided powers' is quite simple.
\begin{lemma} \label{lem155} 
		Let \(A\) be an \(\F_p\)-algebra. Then \(a \in A\) admits nilpotent divided powers if and only if \(a^p = 0.\) 
\end{lemma}
\begin{proof}
		If \(a\) admits nilpotent divided powers \((a_1, a_2, ...),\) then there is some \(a_p \in A\) such that
		\[a^p = p! \cdot a_p\]
		(this follows by repeatedly applying the divided power relation \(a_na_m = \binom{n+m}{n}a_{n+m}\) with \(m=1\)). But \(p! = 0\) in \(A,\) so \(a^p = 0 \cdot a_p = 0.\) 

		Conversely, assume \(a^p = 0.\) Then it is easy to check that  
		\[a_n = \begin{cases} a^n/n! & n < p, \\ 0 & n \geq p, \end{cases}\]
		is a nilpotent PD sequence.
\end{proof}

By \autoref{lem155}, the image of the quasi-ideal \(\G_a^{\#, \wedge} \to \G_a\) is the ideal \(\alpha_p \inclusion \G_a.\) Thus, if instead of quotienting by a groupoid we wanted to quotient by an equivalence relation, then we should transmute by \(\Cone(\alpha_p \inclusion \G_a)\) instead of \(\G_a^{\hat{\dR}} = \Cone(\G_a^{\#, \wedge} \to \G_a).\)

\begin{remark}[The Frobenius twist of a \(\G_a\)-module]
		By a \emph{\(\G_a\)-module}, we mean an fpqc sheaf of abelian groups on \(\Alg_k\) endowed with an action of the abelian sheaf \(\G_a.\) For a \(\G_a\)-module \(M,\) we write \(F_*M\) to denote the \(\G_a\)-module with the same underlying abelian sheaf, but with the action twisted by Frobenius: on \(A\)-points, for \(\lambda \in A, m \in M(A),\) we define 
		\[\lambda \cdot m := \lambda^pm,\]
		where the left hand side is the \(\G_a\)-action on \(F_*M,\) and the right hand side is evaluated using the \(\G_a\)-action on \(M.\)
\end{remark} 
\begin{lemma} \label{lem336}
		We have \(\Cone(\alpha_p \inclusion \G_a) \heq F_*\G_a\) as a \(\G_a\)-module.
\end{lemma}
\begin{proof}
		The commutative square
		\begin{center}
				\begin{tikzcd}
						\alpha_p \ar[r] \ar[d] & 0 \ar[d] \\
						\G_a \ar[r, "x \mapsto x^p"] & F_*\G_a
				\end{tikzcd}
		\end{center}
		of \(\G_a\)-modules induces an injective \(\G_a\)-linear map
		\begin{equation} \label{eq346} \Cone(\alpha_p \inclusion \G_a) \to F_*\G_a.\end{equation}

		This is surjective as a map of sheaves, because locally in the flat topology we can find \(p^{\text{th}}\) roots. 
\end{proof}
\begin{remark}
		In particular, the transmutation of \(S\) with respect to \(\G_a/\alpha_p\) is just \(S',\) the relative Frobenius twist of \(S\) over \(k.\) The natural map \(\G_a^{\hat{\dR}} \to \G_a/\alpha_p\) of ring stacks induces via transmutation a morphism
		\begin{equation} \label{eq464} (S/k)^{\hat{\dR}} \to S'.\end{equation} 
\end{remark}

We now study how far \autoref{eq464} is from being an equivalence. This boils down to understanding how far apart \(\alpha_p\) and \(\G_a^{\#, \wedge}\) are.
\begin{lemma} \label{lem493} 
		There is a short exact sequence of \(\G_a\)-modules
		\begin{equation} \label{eq495} 0 \to F_*\G_a^{\#, \wedge} \to \G_a^{\#, \wedge} \to \alpha_p \to 0.\end{equation}
		The map \(F_*\G_a^{\#, \wedge} \to \G_a^{\#, \wedge}\) is given on \(A\)-points by the formula
		\[(a_0, a_1, ...) \mapsto (b_0 = 1, b_1 = 0, b_2, b_3, ...),\]
		where 
		\[b_n := \begin{cases} 0 & p \nmid n, \\ a_{n/p} & p \mid n.\end{cases}\] 
\end{lemma}
\begin{remark}
		The map \(F_*\G_a^{\#, \wedge} \to \G_a^{\#, \wedge}\) is reminiscient of the Verschiebung; this analogy is made a little more precise by the Witt vector model of \(\G_a^{\hat{\dR}, c}\); see section 2.7 of Bhatt \cite{bhatt} for more information. 
\end{remark}
\begin{proof}
		By Lucas' theorem that \(\binom{pn}{pm}\equiv \binom{n}{m}\pmod{p},\) it's easy to check that \(b_{\bullet}\) is a PD sequence whenever \(a_{\bullet}\) is.
	
		This map \(F_*\G_a^{\#, \wedge} \to \G_a^{\#, \wedge}\) is injective (because a PD sequence \(a_{\bullet}\) in characteristic \(p\) is determined by its values \(a_{pn}\) on the multiples of \(p\)), and \(\G_a\)-linear (the Frobenius twist on the source corrects the shift by \(p\)). 

		As \(a_{\bullet} \in F_*\G_a^{\#, \wedge}(A)\) is always sent to a \(b_{\bullet}\) with \(b_1 = 0,\) we see that this map has image contained in the kernel of \(\G_a^{\#, \wedge} \to \alpha_p.\) It also surjects onto the kernel; for any PD structure \((1, 0, \epsilon_2, ...)\) on 0, the PD structure \((1, \epsilon_p, \epsilon_{2p}, ...)\) will be sent to \((1, 0, \epsilon_2, ...).\) The surjectivity of \(\G_a^{\#, \wedge} \to \alpha_p\) is \autoref{lem155}, so we conclude.
\end{proof}

\begin{defn}
		We set \(\G_a^{\psi_p} := \G_a \oplus B(F_*\G_a^{\#, \wedge}),\) and write \(S^{\psi_p}\) to denote the transmutation of \(S\) with respect to \(\G_a^{\psi_p}.\) 
\end{defn}

This \(S^{\psi_p}\) helps us explain the difference between \((S/k)^{\hat{\dR}}\) and \(S'.\) We can extract a useful invariant of connections, called \emph{\(p\)-curvature}, out of \(S^{\psi_p}.\) 
\begin{remark} \label{rem357} 
		Suppose \(S\) is a smooth \(k\)-scheme. A quasicoherent sheaf on \(S^{\psi_p}\) is then, by deformation theory (note \(\G_a^{\psi_p}\) is a \emph{split} square-zero extension of \(\G_a\)), just a quasicoherent sheaf \(\mathcal{E}\) on \(S\) together with a map \(\psi_p : \mathcal{E} \to F_{\abs}^*\Omega^1_S \otimes \mathcal{E}\) such that \(\psi_p \wedge \psi_p = 0.\) We give a slightly more concrete way to see this for \(S = \A^1\) in \autoref{rem431}.
\end{remark}
\begin{remark} \label{rem431}
		Let \(M\) be a \(k[x]\)-module. We now spell out what a \(F_*\G_a^{\#, \wedge}\)-equivariant structure on \(M\) is. The description is very close to the description of \(\G_a^{\#, \wedge}\)-equivariant structures from \autoref{art217}. 

		To give such an equivariant structure it is equivalent to give \(M\) the structure of a continuous \(H = k[x_1, x_2]\inner{x_2 - x_1}/(x_2 - x_1)\)-comodule, where \(H\) denotes the Hopf algebra 
		\[k\left[\frac{(x_2 - x_1)^p}{p!}, \frac{(x_2-x_1)^{2p}}{(2p)!}, ...\right].\]

		This comodule structure takes the form of a map
		\[\rho : M \to H \hat{\otimes}_{k[x]} M,\]
		and so we may write
		\[\rho(s) = \sum_{n=0}^{\infty} \frac{(x_2-x_1)^{np}}{(np)!} \otimes \psi_n(s)\]
		for various set-theoretic functions \(\psi_n : M \to M.\) The comodule axioms then become \(\psi_0 = \id, \psi_{n+m} = \psi_n \circ \psi_m,\) so that everything is determined by the endomorphism \(\psi : M \to M.\) It's easy to see every possible \(\psi\) gives rise to such a comodule, so that \(F_*\G_a^{\#, \wedge}\)-equivariant structures are the same as endomorphisms of \(M.\) 
\end{remark}
\begin{construction}[The \(p\)-curvature] \label{const360} 
		Suppose \(S\) is a smooth \(k\)-scheme. Let \((\mathcal{E}, \nabla)\) be a vector bundle with flat connection on \(S.\) Viewing \((\mathcal{E}, \nabla)\) as a vector bundle on \((S/k)^{\hat{\dR}},\) we define the \emph{\(p\)-curvature} of \((\mathcal{E}, \nabla)\) to be its pullback along the map
		\[S^{\psi_p} \to (S/k)^{\hat{\dR}}\]
		obtained by transmuting \(\G_a^{\psi_p} \to \G_a^{\hat{\dR}}.\) In light of \autoref{rem357}, we can view a vector bundle on \(S^{\psi_p}\) as a vector bundle on \(S\) together with a flat \(F_{\abs}^*\Omega^1_S\)-twisted endomorphism. As the quotient map \(\G_a \to \G_a^{\hat{\dR}}\) factors through \(\G_a^{\psi_p}\) (since \(\G_a^{\psi_p}\) is defined by quotienting \(\G_a\) by a smaller quasi-ideal), the pullback of \((\mathcal{E}, \nabla)\) to \(S^{\psi_p}\) has underlying vector bundle \(\mathcal{E},\) and hence the \(p\)-curvature can be viewed as a morphism
		\[\psi_p : \mathcal{E} \to F_{\abs}^*\Omega^1_S \otimes \mathcal{E}.\] 
\end{construction}

\begin{prop} \label{prop380}
		Let \(S\) be a smooth \(k\)-scheme. Recall that Cartier descent furnishes \(F_{\abs}^*\Omega^1_S\) with a natural flat connection. For any vector bundle with flat connection \((\mathcal{E}, \nabla),\) the \(p\)-curvature 
		\[\psi_p : \mathcal{E} \to F_{\abs}^*\Omega^1_S \otimes \mathcal{E}\]
		is horizontal. In fact, \(\psi_p\) agrees with the classical construction of \(p\)-curvature (as given by Katz \cite{Katz1972}). 
\end{prop}
\begin{remark}
		The observation that \(p\)-curvature is horizontal seems to have first been made by Ogus \cite{ogus}. A different proof of this fact, also using de Rham stack's, can be found in Section 3 of \cite{mine}.
\end{remark}
\begin{proof}
		Checking a map is horizontal (or, for the latter part, that two maps are equal) can be done etale locally, so we may assume \(S = \A^d_k.\) In fact, by looking at the maps \(\psi_{p, d/dx_i} : \mathcal{E} \to \mathcal{E}, \nabla_{d/dx_i} : \mathcal{E} \to \mathcal{E},\) we can check this one coordinate at a time, and hence we can even assume \(d = 1\) to make the notation simpler.

		In the case where \(S = \Spec k[x],\) we then just check directly, using the explicit identification of connections and \(\G_a^{\#, \wedge}\)-equivariant structures explained in the proof of \autoref{art217}. Specifically, let \((M, \nabla)\) be a vector bundle with flat connection on \(S.\) The corresponding \(\G_a^{\#, \wedge}\)-equivariant structure on \(M\) is the map
		\[\rho : M \to k[x, x^2/2!, ...] \hat{\otimes}_{k[x]} M\]
		defined by
		\[\rho(s) = \sum_{n=0}^{\infty} \frac{x^n}{n!} \otimes \nabla_{d/dx}^n(s).\]

		Restricting this to \(F_*\G_a^{\#, \wedge}\) corresponds to quotienting \(k[x, x^2/2!, ...]\) by the ideal \((x),\) so that only the terms \(x^{np}/(np)!\) survive. As in \autoref{rem431}, we see that the endomorphism \(\psi_{p, d/dx} : M \to M\) encoding the \(F_*\G_a^{\#, \wedge}\)-equivariant structure is just \(\nabla_{d/dx}^p,\) exactly recovering the usual formula for \(p\)-curvature, and so we conclude.
\end{proof}

There is also a more precise claim about the relationship between \((S/k)^{\hat{\dR}} \to S'\) that one can make.
\begin{remark}[The torsor property] \label{torsorproperty}
		The map \((S/k)^{\hat{\dR}} \to S'\) constructed in \autoref{eq464} is in fact a torsor for \(B\mathbb{V}(F_{S/k}^*T_{S'/k}),\) as shown in Proposition 2.7.1 of Bhatt's \cite{bhatt}; though Bhatt's argument was given in the non-sheared case, it goes through essentially verbatim in the sheared case, as explained in section 2.3 of \cite{mine}.
\end{remark}

\section{Non-abelian Griffiths filtered connections} \label{sec3}

Fix \(k\) a field of characteristic \(p,\) and \(S\) a smooth \(k\)-scheme.

As remarked in the introduction, one of the most important structures on algebraic de Rham cohomology is its Hodge filtration. The Hodge filtration interacts with the Gauss--Manin connection via a condition called Griffiths transversality. 

\begin{remark}[Griffiths transversality] 
		We recall that, for a vector bundle with flat connection \((\mathcal{E}, \nabla : \mathcal{E} \to \Omega^1_S \otimes \mathcal{E})\) on \(S,\) a decreasing filtration \(F^{\bullet}\mathcal{E}\) is said to obey \emph{Griffiths transversality} if \(\nabla(F^i) \subseteq \Omega^1_S \otimes F^{i-1}.\) 
\end{remark}

In this section (\autoref{defn568}), we give the non-abelian analogue of a `filtration obeying the Griffiths transversality condition.' In \autoref{katz}, we will see that the non-abelian Gauss--Manin connection admits a non-abelian Hodge filtration obeying Griffiths transversality, which is why we need to develop this notion.

\subsection{Filtered stacks}

To begin, we recall Simpson's theory of non-abelian filtrations.
\begin{notn}
		We write \(\A^1/\G_m\) for the quotient stack of \(\Spec k[\lambda]\) by \(\Spec k[\lambda^{\pm 1}],\) with the usual action. We recall that \(\A^1/\G_m\) has two physical points: an open immersion \(\Spec(k) \inclusion \A^1/\G_m\), corresponding to the open orbit of all nonzero points, and a closed immersion \(B\G_m \to \A^1/\G_m,\) coming from the \(\G_m\)-equivariant closed immersion \(\{0\} \inclusion \A^1.\) 
\end{notn}

Inspired by the Rees construction, which identifies filtered vector spaces with quasicoherent sheaves on \(\A^1/\G_m,\) Simpson made the following definition.
\begin{defn} \label{defn400}
		A \emph{filtration} on a prestack \(T\) is a map of prestacks \(\tilde{T} \to \A^1/\G_m\) together with an isomorphism \(T \heq \tilde{T} \times_{\A^1/\G_m} \Spec(k).\) The \emph{associated graded} of a filtration \(\tilde{T}\) is the prestack \(\gr(\tilde{T}) := \tilde{T} \times_{\A^1/\G_m} B\G_m.\) In particular, \(\gr(\tilde{T})\) admits a natural map to \(B\G_m.\)
\end{defn}

\subsection{Review of Simpson's work and the Hodge-filtered sheared de Rham stack} \label{background2}

In this section, we will define a certain stack \((S/k)^{\hat{\dR}, +}\) over \(\A^1/\G_m\), called the \emph{Hodge-filtered sheared de Rham stack} and studied by Drinfeld \cite{prismatization} and Bhatt--Lurie \cite{bhattlurie} (see also Bhatt's \cite{bhatt}), which geometrizes Hodge-filtered de Rham cohomology in the same way that \((S/k)^{\hat{\dR}}\) geometrizes de Rham cohomology.

The idea behind \((S/k)^{\hat{\dR}, +}\) goes back to Simpson \cite{simpson}. Vector bundles on \((S/k)^{\hat{\dR}}\) are the same as vector bundles on \(S\) with flat connection; however, there is a natural 1-parameter deformation of the notion of flat connection. A \emph{\(\lambda\)-connection}, for \(\lambda \in k,\) on a vector bundle \(\mathcal{E}\) on \(S\) is a map
\[\nabla : \mathcal{E} \to \Omega^1_{S/k} \otimes \mathcal{E}\]
obeying the \(\lambda\)-Leibniz rule
\[\nabla(f \cdot s) = f\nabla(s) + \lambda df \otimes s.\] 

The moduli stack of \(\lambda\)-connections then lives over \(\A^1_{\lambda}.\) However, if \(\lambda' = u\lambda\) for a unit \(u \in k^{\times},\) then \(\lambda\)-connections and \(\lambda'\)-connections are in bijection (rescale by \(u\)). Thus it is slightly more natural to view the moduli of \(\lambda\)-connections as living over \(\A^1/\G_m.\) 

We now define \((S/k)^{\hat{\dR}, +}\) as some modification of the de Rham stack which lives over \(\A^1/\G_m\) and geometrizes \(\lambda\)-connections. 

\begin{exmp}[The Hodge-filtered de Rham stack of \(\A^1\)] \label{exmp382} 
		In the case of \(S = \A^1,\) recall that we defined \((\A^1_k/k)^{\hat{\dR}}\) as the fpqc stack whose functor of points sent a semiperfect \(k\)-algebra \(A\) to the groupoid whose objects are elements \(x \in A,\) and whose isomorphisms \(x \heq y\) are nilpotent PD structures on \(y-x.\) 

		To get \(\lambda\)-connections instead of connections, one just slightly tweaks this groupoid: for any \(\lambda \in k,\) we define \((\A^1_k/k)^{\lambda-\hat{\dR}}\) as the fpqc stack sending a semiperfect \(k\)-algebra \(A\) to the groupoid whose objects are again elements \(x \in A,\) but where isomorphisms \(x \heq y\) are now the datum of a nilpotent PD sequence \((\epsilon_0, \epsilon_1, \epsilon_2, ...)\) in \(A\) such that \(\lambda \epsilon_1 = y - x.\) Thus
		\[(\A^1/k)^{\lambda-\hat{\dR}} \heq \Cone(\G_a^{\#, \wedge} \xto{\lambda} \A^1).\]
\end{exmp}
\begin{defn}
		Inspired by \autoref{exmp382}, we define \((\G_a/k)^{\hat{\dR}, +}\) as the ring stack
		\begin{equation} \label{eq309} (\G_a/k)^{\hat{\dR}, +} := \Cone(\G_a^{\#, \wedge}(-1) \to \A^1)\end{equation} 
		over \(\A^1/\G_m.\)
\end{defn}
\begin{remark}
		Let us explain what we mean by this twist by \(-1.\) It is a consequence of working over \(\A^1/\G_m.\) 

		First, we explain a simpler sort of twist. When defining a stack over \(\A^1,\) your test objects are now \(k[\lambda]\)-algebras instead of \(k\)-algebras. A \(k[\lambda]\)-algebra can be viewed as a \(k\)-algebra \(A\) together with a map \(\lambda : A \to A.\) The twist by \(-1\) indicates that the map \(\G_a^{\#, \wedge}(A) \to A\) which we quotient by to get our groupoid of points \(\Spec(A) \to (\A^1/k)^{\hat{\dR}, +}\) should not just be the canonical forgetful map, but instead should be \(\lambda\) times the canonical forgetful map. 

		To get our meaning of twist, recall that a map \(\Spec(A) \to \A^1/\G_m\), for a \(k\)-algebra \(A,\) is just the data of some \(L \in \Pic(A)\) together with an \(A\)-linear map \(t : L \to A.\) By \autoref{eq309}, we just mean that the functor of points of \((\G_a/k)^{\hat{\dR}, +},\) as an \(\A^1/\G_m\)-stack, is such that for a semiperfect \(k\)-algebra \(A\) with a map \(\Spec(A) \to \A^1/\G_m,\) the mapping space
		\[\Map_{\A^1/\G_m}((\Spec(A), L \in \Pic(A), t : L \to A), (\G_a/k)^{\hat{\dR}, +})\]
		is the groupoid whose objects are elements \(x \in A\) and where an isomorphism \(x \heq y\) is the data of some \(\epsilon \in \G_a^{\#, \wedge}(A) \otimes_A L\) such that \((d \otimes t)(\epsilon) = y - x,\) for \(d : \G_a^{\#, \wedge}(A) \to A\) the forgetful map.  
\end{remark}
\begin{defn}[Hodge-filtered sheared de Rham stacks in general]
		Via transmutation, we then (following Bhatt \cite{bhatt}, definition 2.5.1) define \((S/k)^{\hat{\dR}, +} \to \A^1/\G_m\) as the functor
		\[(S/k)^{\hat{\dR}, +}(\Spec(A) \to \A^1/\G_m) := S(\G_a^{\hat{\dR}, +}(\Spec(A) \to \A^1/\G_m)).\] 
\end{defn}
\begin{remark} \label{rem551}
		The \((S/k)^{\hat{\dR}, +}\to \A^1/\G_m\) gives us a filtration on \((S/k)^{\hat{\dR}}\) in the sense of \autoref{defn400}, because \(\G_a^{\hat{\dR}, +}|_{\lambda=1} = \G_a^{\hat{\dR}}.\) 
\end{remark}
\begin{remark}[Relative Hodge-filtered de Rham stacks]
		For a smooth morphism \(X \to S,\) we generally define 
		\[(X/S)^{\hat{\dR}, +} = (X/k)^{\hat{\dR}, +} \times_{S^{\hat{\dR}, +}} S.\]
\end{remark}

\subsection{The Hodge stack} \label{background3}

The Hodge-filtered sheared de Rham stack \((S/k)^{\hat{\dR}, +}\) lives over \(\A^1/\G_m.\) The stack \(\A^1/\G_m\) has two physical points, \(\Spec(k)\) and \(B\G_m.\) We already saw (in \autoref{rem551}) that base changing along \(\Spec(k)\) turns the Hodge-filtered sheared de Rham stack into the ordinary sheared de Rham stack; in effect, it forgets the filtration.

It is then natural to ask what the base change over \(B\G_m\) is. Viewing a quasicoherent sheaf on \(\A^1/\G_m\) as a filtered vector space, the pullback to \(B\G_m\) corresponds to taking the associated graded (in fact, quasicoherent sheaves on \(B\G_m\) are just \(\Z\)-graded \(k\)-vector spaces). Thus we expect that the base change of \((S/k)^{\hat{\dR}, +}\) to \(B\G_m\) will have something to do with the associated graded of the Hodge filtration. This inspires the following name.
\begin{defn}
		We define the \emph{Hodge stack} of \(S\) over \(k\) to be
		\[(S/k)^{\Hodge} := (S/k)^{\hat{\dR}, +} \times_{\A^1/\G_m} B\G_m.\]
\end{defn}

Remarkably, \((S/k)^{\Hodge}\) has a very simple geometric structure.
\begin{prop} \label{prop555}
		There is an equivalence
		\[(S/k)^{\Hodge} \heq B_{S\times B\G_m}(T_{S/k}(-1)).\]
		In other words, \((S/k)^{\Hodge}\) is a split gerbe over \(S\times B\G_m.\) Here, by \(T_{S/k}(-1),\) we denote the vector bundle \(T_{S/k} \boxtimes \OO(-1)\) on \(S \times B\G_m\) (where \(\OO(-1)\) is the dual of the canonical line bundle \(\OO(1)\) on \(B\G_m\)). 
\end{prop}
\begin{proof}
		We use the argument of Bhatt \cite{bhatt}, footnote 18. Over \(B\G_m,\) the map \(\G_a^{\#, \wedge}(-1) \to \G_a\) is just the zero map, and hence
		\[\G_a^{\hat{\dR}}|_{B\G_m} = \G_a \oplus B\G_a^{\#, \wedge}(-1).\]
		This is the split square-zero extension of \(\G_a\) by \(B\G_a^{\#, \wedge}(-1),\) and so we conclude by deformation theory. 
\end{proof}

\subsection{Griffiths-filtered non-abelian connections}

We can now introduce the notion of a \emph{Griffiths-filtered} non-abelian connection alluded to in the introduction to this section. 

\begin{defn} \label{defn568}
		A \emph{Griffiths-filtered non-abelian connection} over \(S\) is a Cartesian diagram of prestacks
		\begin{center}
				\begin{tikzcd}
						\tilde{E} \ar[r] \ar[d] & \tilde{E}' \ar[d] \\
						S \times \A^1/\G_m \ar[r] & (S/k)^{\hat{\dR}, +}. 
				\end{tikzcd}
		\end{center}

		Base changing this diagram of \(\A^1/\G_m\)-stacks along the open immersion \(\Spec(k) \to \A^1/\G_m,\) we get a Cartesian diagram
		\begin{center}
				\begin{tikzcd}
						E \ar[r] \ar[d] & E' \ar[d] \\
						S \ar[r] & (S/k)^{\hat{\dR}},
				\end{tikzcd}
		\end{center}
		which we call the \emph{underlying non-abelian connection}. 
\end{defn}	
\begin{remark}
		The data of this Cartesian diagram is a bit redundant; one can recover it from just the map \(\tilde{E}' \to (S/k)^{\hat{\dR}, +}.\) 
\end{remark}

We end this section by explaining a certain useful invariant of Griffiths-filtered connections. 
\subsection{Digression: A non-abelian variant of Cartier duality}

Recall that \emph{Cartier duality} identifies quasicoherent sheaves on \(B\mathbb{V}(\mathcal{E})^{\#, \wedge}\) with quasicoherent sheaves on \(\mathbb{V}(\mathcal{E}^{\vee})\) (we don't state this more precisely because we will not use this claim in the article). As the former objects are simpler than the latter, Cartier duality is incredibly computationally useful. 

We now introduce a (weak) non-abelian variant of Cartier duality. One reason for its relevance has already been seen: above, we showed \((S/k)^{\Hodge} \heq B_S\mathbb{V}(\Omega^1_S)^{\#, \wedge}.\) As we are considering non-abelian connections in this paper, instead of connections, we will have prestacks over \(B_S\mathbb{V}(\Omega^1_S)^{\#, \wedge}\) instead of quasicoherent sheaves on it. We now explain a simple construction to extract an easily understood invariant from a prestack over \(B_S\mathbb{V}(\Omega^1_S)^{\#, \wedge}.\) This construction will be used again in the sequel, which is why we give it in so much generality now. 

We first introduce some notation. 
\begin{defn}
		Let \(F : \Sch_S^{\op} \to \Gpd\) be a prestack over \(S.\) Its \emph{tangent prestack} is the functor
		\[\mathcal{T}_{F/S} : \Sch_S^{\op} \to \Gpd\]
		given by \(\mathcal{T}_{F/S}(Z) = F(Z[\epsilon]),\) for \(Z[\epsilon]\) the trivial square-zero extension of \(Z\) by \(\OO_Z.\) 

		Note that there is a natural map \(\mathcal{T}_{F/S} \to F,\) corresponding to the closed immersion \(Z \inclusion Z[\epsilon].\) 
\end{defn} 

We now give two constructions. 
\begin{situation} \label{setup580} 
		Fix \(\mathcal{E}\) a vector bundle on \(S,\) and \(F : \Sch_S^{\op} \to \Gpd\) a prestack, with structure morphism \(p : F \to S.\) 

		Set \(S' = S \oplus \mathcal{E}\) the trivial square-zero extension of \(S\) by \(\mathcal{E},\) and \(F \oplus p^*\mathcal{E}\) the trivial square-zero extension of \(F\) by \(p^*\mathcal{E}.\) That is, \(F \oplus p^*\mathcal{E}\) is the prestack 
		\[Z \mapsto \Hom_S(Z ,S') \times F(S').\] 
\end{situation}
\begin{construction}[Infinitesimal automorphism to derivation] \label{const638}
		Suppose we are in \autoref{setup580}. Fix an automorphism of prestacks
		\[\eta : F \oplus p^*(\mathcal{E}) \heq F \oplus p^*(\mathcal{E}).\]
		Given this automorphism \(\eta,\) we will construct a map of \(S\)-prestacks
		\[\theta : \mathbb{V}(\mathcal{E}^{\vee}) \times_S F \to \mathcal{T}_{F/S}.\]

		Indeed, fix a test scheme \(s : Z \to S.\) We need to produce a map
		\[\theta(Z) : (\mathbb{V}(\mathcal{E}^{\vee}) \times_S F)(Z) \to F(Z[\epsilon]).\]
		The groupoid \((\mathbb{V}(\mathcal{E}^{\vee}) \times_S F)(Z)\) has as its objects the pairs \((\alpha : s^*\mathcal{E} \to \OO_Z, x \in F(Z))\) and has isomorphisms \((\alpha, x) \heq (\alpha, x')\) in bijection with isomorphisms \(x \heq x'\) in \(F(Z).\) 
		On objects, we define \(\theta(Z)(\alpha, x)\) as follows. Writing \(Z' := Z \oplus s^*(\mathcal{E})\) the trivial square-zero extension, we can combine \(\alpha, x\) to get a composite map
		\[x_{\alpha} : Z[\epsilon] \xto{\alpha} Z' \xto{x'} F \oplus p^*(\mathcal{E}).\]
		Here, recalling 
		\[Z' = Z \times_S S', F \oplus p^*(\mathcal{E}) = F \times_S S',\]
		we set \(x' = (x, \id_{S'}).\) 

		We then define
		\[\theta(Z)(\alpha, x) := \pi(\eta(x_{\alpha})),\]
		for \(\pi : F \oplus p^*(\mathcal{E}) \to F\) the projection. As \(\theta(Z)\) is defined by taking \(x\) and applying various compositions, isomorphisms \(x \heq x'\) can be mapped to isomorphisms in \(F(Z[\epsilon]),\) and similarly we can construct coherences proving \(\theta(Z)\) is functorial in \(Z.\) Thus we get our map \(\theta,\) as desired.
\end{construction}
\begin{remark}
		When \(\eta = \id,\) the map \(\theta\) is just the canonical map
		\[\mathbb{V}(\mathcal{E}^{\vee}) \times_S F \to F \inclusion \mathcal{T}_{F/S},\]
		where the second map is the zero-section arising from the natural map \(Z[\epsilon] \xto{\epsilon = 0} Z\) of \(S\)-schemes. This is because we always have a commutative diagram
		\begin{center}
				\begin{tikzcd}
						Z' \ar[r, "x'"] \ar[d] & F \oplus p^*(\mathcal{E}) \ar[d, "\pi"] \\
						Z \ar[r, "x"] & F.
				\end{tikzcd}
		\end{center}
\end{remark}

\begin{construction} \label{const205}
		Fix \(S\) an arbitrary prestack, and choose a vector bundle \(\mathcal{E}\) on \(S.\) Let 
		\begin{center}
				\begin{tikzcd}
						E \ar[r] \ar[d, "\phi"] & E' \ar[d] \\
						S \ar[r, "s"] & B_S\mathbb{V}(\mathcal{E})^{\#, \wedge}
				\end{tikzcd}
		\end{center}
		be a Cartesian diagram of \(S\)-prestacks. Here, \(s : S \to B_S\mathbb{V}(\mathcal{E})^{\#, \wedge}\) is the natural section of the structure map \(B_S\mathbb{V}(\mathcal{E})^{\#, \wedge}) \to S\) (recalling that maps to \(B_S\mathbb{V}(\mathcal{E})^{\#, \wedge}\) classify \(\mathbb{V}(\mathcal{E})^{\#, \wedge}\)-torsors, this section \(s\) just classifies \(\mathbb{V}(\mathcal{E})^{\#, \wedge}\) itself, viewed as a torsor over \(S\)). 

		In this situation, we can construct a map of prestacks
		\[\theta : \mathbb{V}(\phi^*\mathcal{E}) \to \mathcal{T}_{E/S}\]
		as follows.

		This diagram gives rise to a \(\mathbb{V}(\mathcal{E})^{\#, \wedge}\)-action on \(E\), with action map \(a : \mathbb{V}(\mathcal{E})^{\#, \wedge}\times_S E \to E\) defined as the composition
		\begin{align*}
				\mathbb{V}(\mathcal{E})^{\#, \wedge} \times_S E &\heq (S \times_{B\mathbb{V}(\mathcal{E})^{\#, \wedge}} S) \times_S (S \times_{B\mathbb{V}(\mathcal{E})^{\#, \wedge}} E') \\
																&\heq S \times_{B\mathbb{V}(\mathcal{E})^{\#, \wedge}} E' \times_{B\mathbb{V}(\mathcal{E})^{\#, \wedge}} S \\
																&\to S \times_{B\mathbb{V}(\mathcal{E})^{\#, \wedge}} E' \\
																&\heq E. 
		\end{align*}

		We first construct, from this action, an automorphism \(\eta\) of \(E \oplus \phi^*(\mathcal{E}^{\vee})\) (the base change of \(E\) along \(S \oplus \mathcal{E}^{\vee} \to S\)). 

		The groupoid of \(Z\)-points of \(E \oplus \phi^*(\mathcal{E}^{\vee})\) is \(\Hom_S(Z, S \oplus \mathcal{E}^{\vee}) \times E(Z).\) As \(S \oplus \mathcal{E}^{\vee}\) is just the zero section of \(\mathbb{V}(\mathcal{E}),\) we get a closed immersion \(i : S \oplus \mathcal{E}^{\vee} \inclusion \mathbb{V}(\mathcal{E})^{\#, \wedge}.\) Base changing this closed immersion along \(E\) and composing with the action map \(a\) from above, we get a map
		\[E \oplus \phi^*(\mathcal{E}^{\vee}) \inclusion \mathbb{V}(\mathcal{E})^{\#, \wedge} \times_S E \xto{a} E.\]

		This map, together with the projection \(E \oplus \phi^*(\mathcal{E}^{\vee}) \to S \oplus \mathcal{E}^{\vee},\) defines a nontrivial map
		\[\eta : E \oplus \phi^*(\mathcal{E}^{\vee}) \to E\oplus \phi^*(\mathcal{E}^{\vee}).\]
		An inverse to \(\eta\) can be obtained via a similar construction, and so \(\eta\) is an automorphism. In particular, we can now apply \autoref{const638} to construct our desired map \(\theta\). 
\end{construction}

\subsection{Non-abelian Higgs fields} \label{nonabhiggs}

We now explain the non-abelian analogue of the notion of Higgs fields. For the reader's benefit, we start by recalling Higgs fields in the abelian case.
\begin{remark}[What is a Higgs fields in the abelian case?] 
		Suppose we have a vector bundle with flat connection \((\mathcal{E}, \nabla)\) on \(S,\) together with a decreasing filtration \(F^{\bullet}\mathcal{E}\) obeying Griffiths transversality: 
\[\nabla(F^n\mathcal{E}) \subseteq \Omega^1_{S/k} \otimes F^{n-1}\mathcal{E}.\]
As the filtration is decreasing, Griffiths transversality is slightly weaker than saying \(\nabla(F^n\mathcal{E}) \subseteq \Omega^1_{S/k} \otimes F^n\mathcal{E}.\) In other words, \(\nabla\) need not preserve the filtration \(F^{\bullet},\) but it isn't that far from preserving it. The \emph{Higgs field} of \((F^{\bullet}, \mathcal{E}, \nabla)\) is some linear invariant measuring how far \(\nabla\) is from preserving \(F^{\bullet}.\) The Higgs field is defined to be the map
\[\theta : \bigoplus_{n\in\Z} \gr^n_F\mathcal{E} \to \Omega^1_{S/k} \otimes \bigoplus_{n\in\Z} \gr^{n-1}_F\mathcal{E}\]
obtained as the direct sum of the maps
\[\theta_n : \gr^n_F \to \Omega^1_{S/k} \otimes \gr^{n-1}_F\]
given by \(\theta_n(s) = \nabla(s).\) This is well-defined on the graded pieces thanks to Griffiths transversality. Moreover, unlike \(\lambda,\) these maps \(\theta_n\) are actually linear, because
\[\theta_n(fs) = df \otimes s + f\nabla(s) = 0 + f\theta_n(s),\]
because the term \(df \otimes s \in \Omega^1_{S/k} \otimes F^n\) will vanish in \(\Omega^1_{S/k} \otimes \gr^{n-1}_F = \Omega^1_{S/k} \otimes F^{n-1}/F^n.\) 
\end{remark}

To construct a non-abelian analogue of Higgs fields, we simply apply \autoref{const205} to the split gerbe \((S/k)^{\Hodge}.\) 

\begin{construction}[Non-abelian Higgs field] \label{higgsfields}
		Let
		\begin{center}
				\begin{tikzcd}
						E \ar[r] \ar[d] & E' \ar[d] \\
						S \times \A^1/\G_m \ar[r] & (S/k)^{\hat{\dR}, +}
				\end{tikzcd}
		\end{center}
		be a Cartesian diagram (that is, a Hodge-filtered non-abelian connection). This is a Cartesian diagram of stacks over \(\A^1/\G_m\); base change to \(B\G_m\) to get a Cartesian diagram
		\begin{center}
				\begin{tikzcd}
						\gr(E) \ar[r] \ar[d, "\phi"] & E' \times_{\A^1/\G_m} B\G_m \ar[d] \\
						S \times B\G_m \ar[r] & (S/k)^{\Hodge}.
				\end{tikzcd}
		\end{center}
		
		But \((S/k)^{\Hodge}\) is a split gerbe, as we just explained, so \autoref{const205} can be applied to construct a map
		\[\theta : \mathbb{V}(\phi^*T_{S/k}(-1)) \to \mathcal{T}_{\gr(E)/(S\times B\G_m)}.\] 
\end{construction}

\section{Non-abelian conjugate filtrations and \(p\)-curvature} \label{sec4}

Let \(k\) be a field of characteristic \(p\) and \(S\) a smooth \(k\)-scheme. In the introduction, we alluded to the existence of a second filtration on algebraic de Rham cohomology; we now explore it via de Rham stacks.

\subsection{The conjugate filtration}

Algebraic de Rham cohomology in positive characteristic carries an extra filtration besides the Hodge filtration, called the \emph{conjugate filtration}; see the discussion of Katz \cite{Katz1972} for a more classical perspectives.

Bhatt \cite{bhatt} (particularly his section 2.7) takes the following perspective on the conjugate filtration. Recall from \autoref{torsorproperty} that \((S/k)^{\hat{\dR}} \to S'\) is a torsor for \(B\mathbb{V}(\Omega_{S'/k}^1)^{\#, \wedge}.\) It is then natural to put \((S/k)^{\hat{\dR}}\) in a 1-parameter family of stacks over \(S'\) which degenerates to the split torsor. As with the Hodge-filtered de Rham stack, it is convenient to view this degeneration as living over \(\A^1/\G_m\) instead of \(\A^1.\) 

\begin{defn}
		In this construction, we work over \(\A^1/\G_m.\) Following Bhatt \cite{bhatt}'s construction 2.7.8, we define a \(\G_a\)-module \(G_u\) as the pushout
		\begin{center}
				\begin{tikzcd}
						F_*\G_a^{\#, \wedge} \ar[r] \ar[d] & \G_a^{\#, \wedge} \ar[d] \\
						F_*\G_a^{\#, \wedge}(+1) \ar[r] & G_u,
				\end{tikzcd}
		\end{center}
		where \(F_*\G_a^{\#, \wedge} \to F_*\G_a^{\#, \wedge}(+1)\) is the canonical map induced by the twist, and \(F_*\G_a^{\#, \wedge} \to \G_a^{\#, \wedge}\) is the Verschiebung-type map of \autoref{lem493}. 

		The quasi-ideals \(0 : F_*\G_a^{\#, \wedge}(+1) \to \G_a\) and \(\can : \G_a^{\#, \wedge} \to \G_a\) agree on \(F_*\G_a^{\#, \wedge}\) (recall \autoref{lem493}), and hence there is a natural map \(G_u \to \G_a\) making \(G_u\) into a quasi-ideal.
		
		We define \(\G_a^{\hat{\dR}, c} := \Cone(G_u \to \G_a).\) 
\end{defn}
\begin{remark}[Signs]
		When studying the Hodge-filtered de Rham stack, we frequently encountered twists by \(-1.\) This negative one ultiamtely came from the fact that Griffiths transversality deals with \emph{decreasing} filtrations. As the conjugate filtration is instead \emph{increasing}, this \(G_u\) is defined using a twist by \(+1.\) 
\end{remark}
\begin{remark}[An alternative motivation]
		Let \(\lambda\) denote the coordinate on \(\A^1.\) Then, for \(t \in k,\) vector bundles on \(\G_a^{\hat{\dR}, c}|_{\lambda = t}\) are, by the same argument as in \autoref{art217}, naturally identified with vector bundles with flat connection \((\mathcal{E}, \nabla)\) on \(\A^1\) together with a \emph{horizontal} integrable map \(\theta : \mathcal{E} \to \Omega^1_{\A^1/k} \otimes \mathcal{E}\) such that \(t\theta = \psi_p.\) The horizontalness comes from the constraint that the \(\G_a^{\#, \wedge}\) and \(F_*\G_a^{\#, \wedge}\)-equivariant structures commute, just as it did in \autoref{prop380}. Thus this \(\G_a^{\hat{\dR}, c}\) is just the de Rham stack way of accessing the stack Lam--Litt \cite{lamlitt} dubbed \(\mathcal{M}_{\operatorname{conj}}.\)
\end{remark}
\begin{defn}
		We define the \emph{conjugate-filtered de Rham stack} of \(S\) as the stack \((S/k)^{\hat{\dR}, c} \to \A^1/\G_m\) obtained via transmutation with respect to \(\G_a^{\hat{\dR}, c},\) so that
		\[(S/k)^{\hat{\dR}, c}(\Spec(A) \to \A^1/\G_m) = S(\G_a^{\hat{\dR}, c}(\Spec(A) \to \A^1/\G_m)).\]
\end{defn}
\begin{remark} \label{rem689}
		As with the de Rham stack, there are natural maps of \(\A^1/\G_m\)-ring stacks
		\begin{equation} \label{eq681} \G_a \to \G_a^{\hat{\dR}, c} \to F_*\G_a,\end{equation}
		the first one coming from the fact that \(\G_a^{\hat{\dR}, c}\) is a quotient of \(\G_a\) (by the quasi-ideal \(G_u\)) and the second one being the map induced by
		\begin{center}
				\begin{tikzcd}
						G_u \ar[r] \ar[d] & \alpha_p \ar[d] \\
						\G_a \ar[r, "\id"] & \G_a,
				\end{tikzcd}
		\end{center}
		the map \(G_u \to \alpha_p\) being given by the sum of the canonical map \(\G_a^{\#, \wedge} \to \alpha_p\) with the zero map \(0 : F_*\G_a^{\#, \wedge}(+1) \to \alpha_p.\) 

		Transmuting \autoref{eq681}, we get natural maps \(S \times \A^1/\G_m \to S^{\hat{\dR}, c}\) and \(S^{\hat{\dR}, c} \to S' \times \A^1/\G_m,\) 
		whose composition is just the Frobenius \(F_{S/k} \times \id_{\A^1/\G_m} : S \times \A^1/\G_m \to S' \times \A^1/\G_m.\) 
\end{remark}
\begin{remark}[Relative conjugated-filtered de Rham stacks]
		For a smooth morphism \(X \to S,\) we generally define 
		\[(X/S)^{\hat{\dR}, c} = (X/k)^{\hat{\dR}, c} \times_{S^{\hat{\dR}, c}} S.\]
\end{remark}

\subsection{The associated graded of \(S^{\hat{\dR}, c}\)} \label{sec699} 

As with the Hodge-filtered de Rham stack, we can take the associated graded of \(S^{\hat{\dR}, c}.\) 
\begin{defn}
		The \emph{conjugate Hodge stack} is defined as 
		\[S^{\Hodge, c} := S^{\hat{\dR}, c} \times_{\A^1/\G_m} B\G_m.\]
\end{defn}
\begin{remark}This \(S^{\Hodge, c}\) can also be obtained by transmutation with respect to \(\G_a^{\hat{\dR}, c}|_{\lambda = 0}.\)
\end{remark}

\begin{prop} \label{prop750} 
	There is a canonical isomorphism 	
		\[S^{\Hodge, c} \heq B_{S'\times B\G_m}\mathbb{V}(T_{S'/k}(+1)).\] 
\end{prop}
\begin{proof}
Over \(B\G_m,\) the pushout diagram defining \(G_u\) becomes simply
\begin{center}
				\begin{tikzcd}
						F_*\G_a^{\#, \wedge} \ar[r] \ar[d, "0"] & \G_a^{\#, \wedge} \ar[d] \\
						F_*\G_a^{\#, \wedge}(+1) \ar[r] & G_u|_{\lambda = 0},
				\end{tikzcd}
\end{center}
and so \(G_u|_{\lambda = 0} \heq \alpha_p \oplus F_*\G_a^{\#, \wedge}(+1),\) recalling from \autoref{lem493} that \(\G_a^{\#, \wedge}/F_*\G_a^{\#, \wedge} \heq \alpha_p.\) From \autoref{lem336}, we have \(\G_a/\alpha_p \heq F_*\G_a,\) and so
\[\Cone(G_u|_{\lambda =0} \to \G_a) \heq \Cone(F_*\G_a^{\#, \wedge}(+1) \xto{0} F_*\G_a) = F_*\G_a \oplus B(F_*\G_a^{\#, \wedge}(+1)).\]
Exactly as with the Hodge stack, deformation theory now tells us that \(S^{\Hodge, c}\) is a canoncially split \(B\mathbb{V}(T_{S'/k}(+1))\)-torsor over \(S',\) as desired.
\end{proof}

\subsection{Conjugate filtered connections and their \(p\)-curvature}

Analogously to Griffiths-filtered non-abelian connections, we can also define a notion of non-abelian connections equipped with something behaving like the conjugate filtration. 
\begin{defn}
		A \emph{conjugate-filtered non-abelian connection} is a Cartesian diagram of \(\A^1/\G_m\)-prestacks
		\begin{center}
				\begin{tikzcd}
						E \ar[r] \ar[d] & \tilde{E} \ar[d] \\
						S \times \A^1/\G_m \ar[r] & S^{\hat{\dR}, c}.
				\end{tikzcd}
		\end{center}
\end{defn}		

Just as the description of \((S/k)^{\Hodge}\) as a split gerbe allowed us to define non-abelian Higgs fields, \autoref{prop750} allows us to extract a certain interesting invariant out of conjugate-filtered non-abelian connections. 

\begin{remark}[The abelian situation]
		We first explain what the abelian analogue of the following construction is. Suppose \(G_{\bullet}\) is an increasing filtration on a vector bundle \(\mathcal{E}\) endowed with a connection \(\nabla\) such that \(\nabla\) preserves \(G_i\) and such that \(\gr^G(\mathcal{E})\) has \(p\)-curvature zero. In this situation, \(\psi_p\), which of course sends \(G_i\) to \(G_i,\) actually does a little better: it even sends \(G_i\) to \(G_{i+1}.\) Thus we can view \(p\)-curvature as a graded map \(\mathcal{E} \to F_{S/k}^*\Omega^1_{S'/k}(-1) \otimes \mathcal{E},\) with the twist by \(-1\) representing that the grading gets shifted.
\end{remark}
\begin{construction}[Conjugate-filtered \(p\)-curvature] \label{const802}
		Fix a conjugate-filtered non-abelian connection. Base changing to \(B\G_m,\) we get a Cartesian diagram
	\begin{center}
				\begin{tikzcd}
						\gr(E) \ar[r] \ar[d] & \gr(\tilde{E}) \ar[d] \\
						S \times B\G_m \ar[r] & S^{\Hodge, c}.
				\end{tikzcd}
		\end{center}
		We call from \autoref{rem689} that there is a natural map \(\nu : S^{\Hodge, c} \to S' \times B\G_m,\) and the composition \(S \times B\G_m \to S^{\Hodge, c} \to S' \times B\G_m\) is just \((F_{S'/k}, \id_{B\G_m}).\) 

		Recall from \autoref{prop750} that \(S^{\Hodge, c} \to S' \times B\G_m\) is actually a canonically split gerbe. Call the splitting \(s : S' \times B\G_m \to S^{\Hodge, c}.\) Composing this with the map \(S \times B\G_m \to S^{\Hodge, c},\) we get a Cartesian diagram
		\begin{center}
				\begin{tikzcd}
						\gr(E) \ar[r] \ar[d, "\phi"] & \gr(\tilde{E}) \times_{S^{\Hodge, c}} (S' \times B\G_m) \ar[r] \ar[d, "\phi_1"] & \gr(\tilde{E}) \ar[d] \\
						S \times B\G_m \ar[r, "{(F_{S/k}, \id)}"] & S' \times B\G_m \ar[r] \ar[rd, "\id"] & S^{\Hodge, c} \ar[d] \\ 
																  & & S' \times B\G_m
				\end{tikzcd}
		\end{center}
		
		Applying \autoref{const205} to the right square of this diagram, setting \(\gr(E)' := \gr(\tilde{E}) \times_{S^{\Hodge, c}} (S'\times B\G_m),\) we get a map 
		\[\alpha : \mathbb{V}(\phi_1^*T_{S'/k}(+1)) \to \mathcal{T}_{\gr(E)'/(S' \times B\G_m)}.\]

		We then define 
		\[\psi_p : \mathbb{V}(\phi^*F_{S/k}^*T_{S'/k}(+1)) \to \mathcal{T}_{\gr(E)/(S \times B\G_m)}\]
		as the base change of \(\alpha\) along \((F_{S/k}, \id) : S \times B\G_m \to S' \times B\G_m.\) 
\end{construction}

\section{Katz's formula} \label{katz}

Finally, we can now prove Katz's formula. 

\subsection{The isomonodromy foliation}

Fix \(k\) a field of characteristic \(p,\) and \(S\) a smooth \(k\)-scheme. Fix \(f : X \to S\) a smooth proper morphism. 

\begin{warn}
		Throughout this section, we will write \(S^{\hat{\dR}}\) to denote \((S/k)^{\hat{\dR}},\) as everything will be done over \(k.\) 
\end{warn}

Recall from \autoref{defn65} that a non-abelian connection on \(S\) is a stack over \(S^{\hat{\dR}}.\) In the introduction, we defined
\[\mathcal{M}_{\dR}(X/S, n) : \Sch_S \to \Gpd\]
as the functor sending an \(S\)-scheme \(T\) to the groupoid of triples \((\mathcal{E}, \nabla, \xi)\) where
\begin{enumerate}
		\item \(\mathcal{E}\) is a rank \(n\) vector bundle on \(X\times_S T,\) 
		\item \(\nabla : \mathcal{E} \to \Omega^1_{(X\times_S T)/T}\otimes\mathcal{E}\) is a flat connection,
		\item \(\xi : \det(\mathcal{E}) \heq \OO_{X\times_S T}\) is an equivalence.
\end{enumerate} 

This moduli stack \(\mathcal{M}_{\dR}(X/S, n)\) of vector bundles with flat connection (which also appears in Lam--Litt \cite{lamlitt}) will be the underlying object of the non-abelian Gauss--Manin connection.

We asserted in the introduction that this moduli stack was also a certain mapping stack; we now check this assertion. 
\begin{lemma}
		The functor \(\mathcal{M}_{\dR}(X/S, n)\) coincides with the functor of points of the functor
		\[\Map_S((X/S)^{\hat{\dR}}, B\SL_n) : \Sch_S \to \Gpd\]
		sending an \(S\)-scheme \(T\) to the groupoid of \(\SL_n\)-bundles on \((X/S)^{\hat{\dR}} \times_S T.\) 
\end{lemma}
\begin{proof}
		Fix \(T\) an \(S\)-scheme. By definition, a \(T\)-point of the mapping stack is just a map
		\[(X/S)^{\hat{\dR}} \times_S T \to B\SL_n.\] 
		But note (recalling the language of relative sheared de Rham stacks from \autoref{rem139}, and especially \autoref{eq149})
		\begin{align*}
				(X\times_S T/T)^{\hat{\dR}} &= (X \times_S T)^{\dR} \times_{T^{\dR}} T \\
											&= (X^{\dR} \times_S T^{\dR}) \times_{T^{\dR}} T  \\
											&= X^{\dR} \times_S T.
		\end{align*}

		In particular, a map \((X/S)^{\hat{\dR}} \times_S T \to B\SL_n\) is the same as an \(\SL_n\)-bundle on \((X\times_S T/T)^{\hat{\dR}}\); but \(\SL_n\)-bundles on this relative de Rham stack are just rank \(n\) bundles on \(X\times_S T\) with a flat connection relative to \(T\) and with an isomorphism \(\det(\mathcal{E}) \heq \OO\) compatible with connections (because it is really an isomorphism in \(\Qcoh((X/S)^{\hat{\dR}})\)), exactly coinciding with the functor of points of \(\mathcal{M}_{\dR}(X/S, n),\) as desired.
\end{proof}

\begin{defn} \label{defn850}
		The \emph{non-abelian Gauss--Manin connection} is the non-abelian connection furnished by the Cartesian diagram
\begin{center}
		\begin{tikzcd}
				\mathcal{M}_{\dR}(X/S, n) \ar[r] \ar[d] & \Map_{S^{\hat{\dR}}}(X^{\hat{\dR}}, B\SL_n) \ar[d] \\
				S \ar[r] & S^{\hat{\dR}}. 
		\end{tikzcd}
\end{center}
\end{defn}

\begin{remark} \label{rem272} 
		Here, by \(\Map_{S^{\hat{\dR}}}(X^{\hat{\dR}}, B\SL_n),\) we mean the prestack \(\Alg_k \to \Gpd\) where \(\Spec(A)\) is sent to the groupoid of pairs \((p, \mathcal{E})\) where \(p\) is an object of the groupoid \(S^{\hat{\dR}}(A),\) and \(\mathcal{E}\) is an \(\SL_n\)-bundle on the base change of \((X/k)^{\hat{\dR}}\) along \(p : \Spec(A) \to S^{\hat{\dR}}.\)

		An isomorphism \((p, \mathcal{E}) \heq (p', \mathcal{E}')\) in this groupoid is just the data of an isomorphism \(\alpha : p \heq p'\) in \(S^{\hat{\dR}}(A),\) together with an isomorphism \(\beta : \mathcal{E} \heq \tilde{\alpha}^*(\mathcal{E}'),\) where \(\tilde{\alpha} : p^*(X/k)^{\hat{\dR}} \to (p')^*(X/k)^{\hat{\dR}}\) is the isomorphism of stacks induced by \(\alpha.\) 
\end{remark}
\begin{remark}
		We will not use any fpqc descent properties of this mapping prestack, so we do not verify it satisfies any. It seems checking this would be a bit subtle, as \((X/k)^{\hat{\dR}}\) is in general not an algebraic stack, so existing results in the literature on descent properties of mapping prestacks do not seem to apply.
\end{remark}

\subsection{The Hodge filtration} 

\begin{defn} \label{defn871} 
		The \emph{Hodge filtration} on the non-abelian Gauss--Manin connection is the Griffiths-transverse filtration (in the sense of \autoref{defn568}) furnished by the Cartesian diagram. 
\begin{center}
		\begin{tikzcd}
				\Map_{S\times\A^1/\G_m}((X/S)^{\hat{\dR}, +}, B\SL_n) \ar[r] \ar[d] & \Map_{S^{\hat{\dR}, +}}(X^{\hat{\dR}, +}, B\SL_n) \ar[d] \\
				S \times \A^1/\G_m	 \ar[r]	& S^{\hat{\dR}, +}
		\end{tikzcd}
\end{center}
\end{defn}

\begin{defn}
		We define the \emph{Dolbeault stack} as the \(S\times B\G_m\)-stack
		\[\mathcal{M}_{\Dol} := (S\times B\G_m) \times_{S \times (\A^1/\G_m)} \Map_{S\times\A^1/\G_m}((X/S)^{\hat{\dR}, +}, B\SL_n) \heq \Map_{S\times B\G_m}((X/S)^{\Hodge}, B\SL_n).\]
		We write \(\pi : \mathcal{M}_{\Dol} \to S\times B\G_m\) denote the projection. 
\end{defn}
\begin{remark} \label{rem886} 
		We can describe \(\mathcal{M}_{\Dol}\) a little more concretely. As \((X/S)^{\Hodge}\) is a split gerbe over \(X\times B\G_m\) (by \autoref{prop555}), a vector bundle on \((X/S)^{\Hodge}\) is the same as a vector bundle on \(X\times B\G_m\) -- equivalently, a graded vector bundle \(\mathcal{E} = \bigoplus_n \mathcal{E}_n\) on \(X\) -- together with a flat map
		\[\theta : \mathcal{E} \to \Omega^1_{X/S}(+1) \otimes \mathcal{E},\]
		where the \(+1\) shift can be interpreted as saying \(\theta\) sends \(\mathcal{E}_n\) to \(\mathcal{E}_{n-1}.\) 

		Thus a point of \(\mathcal{M}_{\Dol}\) can be thought of as a triple \((\mathcal{E}, \theta, \xi),\) where \(\mathcal{E}, \theta\) are as above, and \(\xi\) is an isomorphism \(\det(\mathcal{E}) \heq \OO\) intertwining \(\det(\theta)\) with zero.
\end{remark}

Applying \autoref{higgsfields} (our construction of non-abelian Higgs fields) to the Hodge filtration (\autoref{defn871}), we produce a morphism 
\begin{equation} \label{eq935} \Theta_{X/S} : \mathbb{V}(\pi^*T_{S/k}(-1)) \to \mathcal{T}_{\mathcal{M}_{\Dol}/(S\times B\G_m)}.\end{equation} 
of \(\mathcal{M}_{\Dol}\)-prestacks.

\subsection{The conjugate filtration}

\begin{defn} \label{defn904} 
		The conjugate filtration on the non-abelian Gauss--Manin connection is the conjugate-filtered connection furnished by the Cartesian diagram
		\begin{center}
		\begin{tikzcd}
				\Map_{S\times\A^1/\G_m}((X/S)^{\hat{\dR}, c}, B\SL_n) \ar[r] \ar[d] & \Map_{S^{\hat{\dR}, c}}(X^{\hat{\dR}, c}, B\SL_n) \ar[d] \\
				S \times \A^1/\G_m	 \ar[r]	& S^{\hat{\dR}, c}.
		\end{tikzcd}
\end{center}
\end{defn} 

Set
\[\mathcal{M}_{\Dol, c} := \Map_{S\times B\G_m}((X/S)^{\Hodge, c}, B\SL_n).\]
the associated graded of the conjugate filtration, with \(\pi_c : \mathcal{M}_{\Dol, c} \to S\times B\G_m\) its structure map. Applying \autoref{const802} to the conjugate filtration (\autoref{defn904}), we get a morphism
\begin{equation} \label{eq950} \psi_{X/S}|_{\lambda = 0} : \mathbb{V}(\pi_c^*F_{S/k}^*T_{S'/k}(+1)) \to \mathcal{T}_{\mathcal{M}_{\Dol, c}/(S\times B\G_m)}.\end{equation} 

\subsection{Identifying the Hodge points}

We are finally ready to state and prove the main theorem.
\begin{theorem}[The non-abelian Katz's formula]
		There is a Cartesian diagram of stacks
		\begin{center}
				\begin{tikzcd}[column sep=large]
						\mathcal{M}_{\Dol, c} \ar[r, "\phi"] \ar[d, "\pi_c"] & \mathcal{M}_{\Dol} \ar[d, "\pi"] \\
						S \times B\G_m \ar[r, "{(F_{\abs}, [-1])}"] & S \times B\G_m,
				\end{tikzcd}
		\end{center}
		and for
		\[\Theta_{X/S} : \mathbb{V}(\pi^*T_{S/k}(-1)) \to \mathcal{T}_{\mathcal{M}_{\Dol}/(S\times B\G_m)}\] 
		as defined in \autoref{eq935} and 
		\[\psi_{X/S}|_{\lambda = 0} : \mathbb{V}(\pi_c^*F_{S/k}^*T_{S'/k}(+1)) \to \mathcal{T}_{\mathcal{M}_{\Dol, c}/(S\times B\G_m)}\]
		as defined in \autoref{eq950}, we have
		\[\phi^*\Theta_{X/S} = \psi_{X/S}|_{\lambda = 0}.\] 
\end{theorem}
\begin{proof}
		Consider the Cartesian diagram
		\begin{center}
				\begin{tikzcd}
						S' \ar[r, "\tau"] \ar[d] & S \ar[d] \\
						\Spec(k) \ar[r, "\Frob"] & \Spec(k)
				\end{tikzcd}
		\end{center}
		defining the relative Frobenius twist; in particular, we let \(\tau : S' \to S\) denote the base change of Frobenius. We can construct an equivalence
		\begin{equation} \label{eq996} (\tau, [-1])^*S^{\Hodge} \heq S^{\Hodge, c}\end{equation}
		as the composition
		\begin{align*}
				(\tau, [-1])^*S^{\Hodge} &\heq (\tau, [-1])^*B_{S\times \G_m}(T_{S/k}(-1)) \\
										 &\heq B_{S'\times B\G_m}(\tau^*T_{S/k}(+1)) \\
										 &\heq B_{S'\times B\G_m}(T_{S'/k}(+1)) \\
										 &\heq S^{\Hodge, c}. 
		\end{align*}

		We can similarly compute that the base change product \((\tau, [-1])^*X^{\Hodge}\) is just
		\begin{equation} \label{eq1013} (\tau, [-1])^*X^{\Hodge} \heq X^{\Hodge, c},\end{equation}
		as
		\begin{align*}
				(\tau, [-1])^*X^{\Hodge} &\heq (\tau, [-1])^*B_{X\times B\G_m}(T_{X/k}(-1))^{\#, \wedge} \\
										 &\heq B_{X' \times B\G_m}(T_{X'/k}(+1))^{\#, \wedge} \\
										 &\heq X^{\Hodge, c}.
		\end{align*}

		Since \((X/S)^{\Hodge} = X^{\Hodge} \times_{S^{\Hodge}} (S \times B\G_m)\) and \((X/S)^{\Hodge, c} = X^{\Hodge, c} \times_{S^{\Hodge, c}} (S \times B\G_m),\) we deduce from \autoref{eq996} and \autoref{eq1013} that the base change, along \((\tau, [-1]),\) of the Cartesian diagram
		\begin{center}
		\begin{tikzcd}
				\Map_{S\times B\G_m}((X/S)^{\Hodge}, B\SL_n) \ar[r] \ar[d] & \Map_{S^{\Hodge}}(X^{\Hodge}, B\SL_n) \ar[d] \\
				S \times B\G_m	 \ar[r]	& S^{\Hodge}
		\end{tikzcd}
		\end{center}
		of \(S\times B\G_m\)-stacks (used to define \(\Theta_{X/S}\)) is just the Cartesian diagram
		\begin{center}
		\begin{tikzcd}
				\Map_{S'\times B\G_m}((X/S)^{\Hodge, c}, B\SL_n) \ar[r] \ar[d] & \Map_{S^{\Hodge, c}}(X^{\Hodge, c}, B\SL_n) \ar[d] \\
				S' \times B\G_m	 \ar[r]	& S^{\Hodge, c}.
		\end{tikzcd}
		\end{center}

		It follows immediately that \((\tau, [-1])^*\Theta_{X/S}\) is the map \(\alpha\) used as an intermediate step in \autoref{const802} to construct \(p\)-curvature; thus 
		\[\psi_{X/S}|_{\lambda =0} = (F_{S/k}, \id)^*\alpha = (F_{S/k}, \id)^*(\tau, [-1])^*\Theta_{X/S} = (F_{\abs}, [-1])^*\Theta_{X/S},\]
		as desired.
\end{proof}

\bibliographystyle{plain}
\bibliography{refs} 

\end{document}